\documentclass[11pt]{amsart}
\usepackage[margin=1.5in]{geometry} 

\usepackage{amssymb,enumitem}
\usepackage{amsmath}
\usepackage{amsthm}
\usepackage{mathrsfs}
\usepackage{hyperref}
\usepackage{amsrefs}
\usepackage{comment}
\usepackage{amsthm,amsrefs,amssymb,amsmath}
\numberwithin{equation}{section}

\newtheorem{thm}{Theorem}[section]
\newtheorem{prop}[thm]{Proposition}
\newtheorem{lemma}[thm]{Lemma}

\theoremstyle{definition}
\newtheorem{definition}[thm]{Definition}
\newtheorem*{example}{Example}

\newcommand{\Z}{\mathbb{Z}}
\newcommand{\N}{\mathbb{N}}
\newcommand{\rt}{\sqrt}

\newcommand{\bd}{\begin{description}}
\newcommand{\ed}{\end{description}}
\newcommand{\ds}{\displaystyle}

\newcommand{\Sym}{\operatorname{Sym}}
\newcommand{\Alt}{\operatorname{Alt}}
\newcommand{\SL}{\operatorname{SL}}

\title[Congruences between the alternating and symmetric groups]{Congruences between word length statistics for the finitary alternating and symmetric groups}
\date{\today}
\author{Tessa Cotron}
\address{605 Asbury Circle, Box 122042, Atlanta, GA 30322}
\email{tessa.cotron@emory.edu}

\author{Robert Dicks}
\address{605 Asbury Circle, Atlanta, GA 30322}
\email{rdicks@emory.edu}

\author{Sarah Fleming}
\address{1192 Paresky Center, Williams College, Williamstown, MA 01267}
\email{smf1@williams.edu}

\begin{document}
\begin{abstract}
In \cite{bacher}, Bacher and de la Harpe study conjugacy growth series of infinite permutation groups and their relationships with $p(n)$, the partition function, and $p(n)_{\textbf{e}}$, a generalized partition function. They prove identities for the conjugacy growth series of the finitary symmetric group and the finitary alternating group. The group theory in~\cite{bacher} also motivates an investigation into congruence relationships between the finitary symmetric group and the finitary alternating group. Using the Ramanujan congruences for the partition function $p(n)$ and Atkin's generalization to the $k$-colored partition function $p_{k}(n)$, we prove the existence of congruence relations between these two series modulo arbitrary powers of 5 and 7, which we systematically describe.  Furthermore, we prove that such relationships exist modulo powers of all primes $\ell\geq 5$.
\end{abstract}
\maketitle

\section{Introduction}

In a recent paper~\cite{bacher}, Bacher and de la Harpe study the conjugacy growth series of infinite permutation groups that are locally finite. Given $g \in G$, where $G$ is a group generated by some set $S$, define the \textit{word length} $\ell_{G,S}(g)$ as the minimal non-negative integer $n$ where $g=s_1s_2\cdots s_n$ and $s_1,s_2,\ldots,s_n\in S\cup S^{-1}$. They define the \textit{conjugacy length} $\kappa_{G,S}(g)$ as the minimal integer $n$ such that there exists an $h$ in the conjugacy class of $g$ for which $\ell_{G,S}(h)=n$. The number of conjugacy classes in $G$ consisting of elements $g$ where $\kappa_{G,S}(g)=n$ for $n\in\N$ is denoted $\gamma_{G,S}(n)\in\N\cup\{0\}\cup\{\infty\}$. If $\gamma_{G,S}(n)$ is finite for all $n\in\N$ for a given pair $(G, S)$, then the \textit{conjugacy growth series} is defined to be
\begin{equation}
C_{G,S}(q):=\sum_{n=0}^{\infty}\gamma_{G,S}(n)q^n.
\end{equation}

Given a permutation $g$ of a non-empty set $X$, denote the \textit{support} of $g$ as $\sup(g):=\{x \in X : g(x)\neq x\}$. The group of permutations of $X$ with finite support is the \textit{finitary symmetric group} $\Sym(X)$. The subgroup of $\Sym(X)$ with permutations of even signature is the \textit{finitary alternating group} $\Alt(X)$.  Let $S\subset\Sym(\N)$ be a generating set of $\Sym(\N)$ such that $S_{\N}^{\text{Cox}}\subset S\subset T_{\N}$, where
\begin{equation}\label{cox}
S_{\N}^{\text{Cox}}:=\{(i,i+1):i\in\N\}
\end{equation}  
is such that $(\Sym(\N), S_{\N}^{\text{Cox}})$ is a Coxeter system, and 
\begin{equation}\label{TN}
T_{\N}:=\{(x,y)\in\Sym(\N): x,y\in\N\text{ are distinct}\}
\end{equation} 
is the conjugacy class of all transpositions in $\Sym(\N)$.  Throughout this paper, we define $S$ to be the set described above.  By Proposition 1 in~\cite{bacher}, the conjugacy growth series for the pair $(\Sym(\N), S)$ is given by 
\begin{equation}\label{csym}
C_{\Sym(\N),S}(q)=\sum_{n=0}^{\infty}p(n)q^n=\prod_{n=1}^{\infty}\frac{1}{1-q^n},
\end{equation}
where $p(n)$ denotes the usual integer partition function.  Let $S'\subset\Alt(\N)$ be a generating set of $\Alt(\N)$ such that $S_{\N}^{A}\subset S'\subset T_{\N}^{A}$, where we define
\begin{equation}\label{SNA}
S_{\N}^A:=\{(i, i+1, i+2)\in\Alt(\N) : i\in N\}
\end{equation}
and 
\begin{equation}\label{TNA}
T_{\N}^A:=\cup_{g\in\Alt(\N)}g S_{\N}^A g^{-1}.
\end{equation}
Throughout this paper, we define $S'$ to be the set described above.  By Proposition 11 in~\cite{bacher}, the conjugacy growth series for the pair ($\Alt(\N), S')$ is given by
\begin{align}\label{calt1}
C_{\Alt(\N),S'}(q)&=\frac{1}{2}\sum_{n=0}^{\infty}p\left(\frac{n}{2}\right)q^n+\frac{1}{2}\sum_{n=0}^{\infty}p_{2}(n)q^n\\
\nonumber&=\frac{1}{2}\prod_{n=1}^{\infty}\frac{1}{1-q^{2n}}+\frac{1}{2}\prod_{n=1}^{\infty}\frac{1}{(1-q^n)^2},
\end{align}
where $p\left(\frac{n}{2}\right)=0$ for all odd $n$ and $p_2(n)$ denotes the number of 2-colored partitions of $n$.  Combining~(\ref{csym}) and~(\ref{calt1}), we obtain
\begin{align}\label{calt2}
2\gamma_{\Alt(\N),S'}(2n)&=p(n)+p_2(2n)\\
\nonumber&=\gamma_{\Sym(\N),S}(n)+p_2(2n).
\end{align}

The finite symmetric and alternating groups, $S_n$ and $A_n$, have the property that $A_n$ is a normal subgroup of $S_n$ and $[S_n:A_n]=2$, so one would naively expect a similar relationship between the finitary symmetric and alternating groups to hold.  By~(\ref{calt2}), we see that $\gamma_{\Alt(\N),S'}(2n)$ is one half of $\gamma_{\Sym(\N),S}(n)$ together with a ``discrepancy function,'' $p_2(2n)$.  In terms of size, we prove in~\cite{cotron} that $\gamma_{\Alt(\N),S'}(n)$ and $\gamma_{\Sym(\N),S}(n)$ behave differently asymptotically and do not follow the pattern of their finite counterparts.  More precisely, in~\cite{cotron}, the authors prove that as $n\rightarrow\infty$, we have that
\begin{equation}
\gamma_{\Sym(\N),S}(n)\sim\frac{e^{\pi\rt{2n/3}}}{4n\rt{3}}
\end{equation}
and
\begin{equation}
\gamma_{\Alt(\N),S'}(n)\sim \frac{e^{2\pi\rt{n/3}}}{3^{\frac{3}{4}}\cdot 8n^{\frac{5}{4}}}.
\end{equation}

It is natural to consider the question of congruence relations between the coefficients of the conjugacy growth series of each of these groups.  By~(\ref{calt2}), there will exist congruences modulo powers of primes between $2\gamma_{\Alt(\N),S'}(2n)$ and $\gamma_{\Sym(\N),S}(n)$  whenever the ``discrepancy function," $p_2(2n)$, is equivalent to 0.  In~\cite{cotron}, we establish a method of proving congruences for generalized partition functions modulo a prime, including the following examples.

\begin{example}
For all $n\equiv 2, 3, 4\pmod{5}$, we have that
\begin{equation*}
2\gamma_{\Alt(\N),S'}(2n)\equiv\gamma_{\Sym(\N),S}(n)\pmod{5}.
\end{equation*}
\end{example}

\begin{example}
For all $n\equiv 17, 31, 38, 45\pmod{49}$, we have that
\begin{equation*}
2\gamma_{\Alt(\N),S'}(2n)\equiv\gamma_{\Sym(\N),S}(n)\pmod{7}.
\end{equation*}
\end{example}

It is natural to ask if these examples are part of a more general phenomenon and if there exists a method to determine congruences.  Ramanujan stated congruences for the partition function $p(n)$ modulo powers of 5, 7, and 11, proved by Watson in~\cite{watson}.  In addition, Atkin proved the existence of congruences for the function $p_2(n)$ modulo powers of the primes 5, 7, and 13 in~\cite{atkin}.  Building off of these results, we obtain congruences between the coefficients of the conjugacy growth series for these groups modulo powers of 5 and 7.  

Throughout, we let $S\subset\Sym(\N)$ be a generating set of $\Sym(\N)$ such that $S_{\N}^{\text{Cox}}\subset S\subset T_{\N}$, where $S_{\N}^{\text{Cox}}$ and $T_{\N}$ are defined by~(\ref{cox}) and~(\ref{TN}), respectively.  In addition, we let $S'\subset\Alt(\N)$ be a generating set for $\Alt(\N)$ such that $S_{\N}^{A}\subset S'\subset T_{\N}^{A}$, where $S_{\N}^{A}$ and $T_{\N}^{A}$ are defined by~(\ref{SNA}) and~(\ref{TNA}), respectively.  Using this notation, we arrive at the following theorem:

\begin{thm}\label{cong1}
Assume the notation above.  Let $\ell=5$ or $7$ and let $j\geq 1$.  Then for all $24n\equiv 1\pmod{\ell^j}$, we have that
\begin{equation*}
\gamma_{\Alt(\N),S'}(2n)\equiv\gamma_{\Sym(\N),S}(n)\equiv 0\pmod{\ell^{\lfloor{j/2-1}\rfloor}}.
\end{equation*}
\end{thm}

\begin{example}
For example, modulo $5$, $25$, and $125$, we obtain for all $n\geq0$ that
\begin{align*}
\gamma_{\Alt(\N),S'}(2\cdot 5^4n+1198)&\equiv
0\pmod{5}\\
\gamma_{\Alt(\N),S'}(2\cdot 5^6n+29948)&\equiv
0\pmod{25}\\
\gamma_{\Alt(\N),S'}(2\cdot 5^8n+748698)&\equiv
0\pmod{125}.
\end{align*}
Likewise, modulo $7$, $49$, and $343$, we obtain for all $n\geq 0$ that
\begin{align*}
\gamma_{\Alt(\N),S'}(2\cdot 7^4n+4602)&\equiv
0\pmod{7}\\
\gamma_{\Alt(\N),S'}(2\cdot 7^6n+225494)&\equiv
0\pmod{49}\\
\gamma_{\Alt(\N),S'}(2\cdot 7^8n+11049202)&\equiv
0\pmod{343}.
\end{align*}
\end{example}

It is a natural question to ask what holds for general primes $\ell\not\in\{5,7\}$.  
Following the work of Treneer~\cite{treneer}, we prove the existence of congruences between the coefficients of the conjugacy growth series for $(\Alt(\N),S')$ and $(\Sym(\N),S)$ modulo arbitrary powers of primes $\ell\geq 5$.  Treneer's work gives general congruences for coefficients of various types of modular forms.  Here, we follow her method and make it explicit.

Let $\ell\geq 5$ be prime.  We then define
\begin{equation}\label{m_l}
m_{\ell}:=
\begin{cases}
2 & 5\leq\ell\leq 23\\
1 & \ell\geq 29,
\end{cases}
\end{equation}
\begin{equation}\label{delta_l}
\delta_{\ell}:=
\frac{Q\ell^{m_{\ell}}\beta_{\ell}+1}{24},
\end{equation}
and
\begin{equation}\label{beta_l}
\beta_{\ell}:=
\frac{23}{Q\ell^{m_\ell}}\pmod{24}.
\end{equation}

Using this notation, we arrive at the following theorem:

\begin{thm}\label{sym}
Assume the above notation.  Let $\ell\geq 5$ be prime and let $j\geq 1$.  Then for a positive proportion of primes $Q\equiv -1\pmod{144\ell^j}$, we have that
\begin{equation*}
2\gamma_{\Alt(\N),S'}(2Q\ell^{m_{\ell}}n+2\delta_\ell)\equiv\gamma_{\Sym(\N),S}(Q\ell^{m_{\ell}}n+\delta_\ell)\pmod{\ell^j}
\end{equation*}
for all $24n+\beta_{\ell}$ coprime to $Q\ell$.
\end{thm}

In this paper, we make effective the effort of~\cite{treneer} by focusing on the specific case of the function $p_2(n)$, which arises from the conjugacy growth series for $(\Alt(\N),S')$.  We use properties of modular forms to prove the existence of congruences between the coefficients of these two series.

Section~\ref{modularity} covers the basics of modular forms and cusps, and Section~\ref{atkinsthm} focuses on congruences for the partition function $p(n)$ and the $k$-colored partition function $p_k(n)$ modulo prime powers.  In Section~\ref{serresthm} we state a theorem of Serre.  Section~\ref{pf} provides a proof of Theorem~\ref{cong1} using the Ramanujan and Atkin congruences.  In Section~\ref{proof of thm}, we focus on and make effective a special case of Theorem 1.2 in~\cite{treneer}, which we require in order to prove the existence of congruences between $2\gamma_{\Alt(\N),S'}(2n)$ and $\gamma_{\Sym(\N),S}(n)$.  In Section~\ref{section4}, we use this result to prove Theorem~\ref{sym}.

\section*{Acknowledgments}
The authors would like to thank Ken Ono and Olivia Beckwith for advising this project and for their many helpful conversations and suggestions. In addition, the authors would like to thank Emory University and the NSF for their support via grant DMS-1250467.

\section{Preliminaries}\label{prelim}

Here we state standard properties of modular forms that can be found in many texts such as~\cite{apostol} and~\cite{ono}.  

\subsection{Modularity, Cusps, and Operators}\label{modularity}

The proof of Theorem~\ref{sym} uses properties of modular forms, which requires the understanding of cusps. A \textit{cusp} of $\Gamma\subseteq\SL_2(\Z)$ is an equivalence class in $\mathbb{P}^1(\mathbb{Q})=\mathbb{Q}\cup\{\infty\}$ under the action of $\Gamma$~\cite[p. 2]{ono}.  We use the following definition of modular forms from~\cite[p. 114]{apostol}: 
\begin{definition}
A function $f$ is said to be an \textit{entire modular form of weight k} on a subgroup $\Gamma\subseteq\SL_{2}(\Z)$ if it satisfies the following conditions: 
\begin{enumerate}[label=(\roman*)]
\item $f$ is analytic in the upper-half $\mathbb{H}$ of the complex plane,
\item $f$ satisfies the equation 
\begin{equation*}
f\left(\frac{az+b}{cz+d}\right)=(cz+d)^kf(z)
\end{equation*} 
whenever $\begin{pmatrix}
a& b\\
c& d
\end{pmatrix}\in\Gamma$ and $z\in\mathbb{H}$, and
\item the Fourier expansion of $f$ has the form 
\begin{equation*}
f(z)= \sum_{n=0}^{\infty} c(n)e^{2\pi i n z}
\end{equation*}
at the cusp $i\infty$, and $f$ has analogous Fourier expansions at all other cusps.
\end{enumerate}
\end{definition}

In particular, we study \textit{Dedekind's eta-function}, a weight 1/2 modular form defined as the infinite product
\begin{equation*}
\eta(z):=q^{1/24}\ds\prod_{n=1}^{\infty}(1-q^n),
\end{equation*}
where $q:=e^{2\pi iz}$ and $z\in\mathbb{H}$.  An \textit{eta-quotient} is a function $f(z)$ of the form
\begin{equation*}
f(z):=\ds\prod_{\delta\mid N}\eta(\delta z)^{r_{\delta}},
\end{equation*}
where $N\geq 1$ and each $r_{\delta}$ is an integer.  If each $r_{\delta}\geq 0$, then $f(z)$ is known as an \textit{eta-product}.

If $N$ is a positive integer, then we define $\Gamma_0(N)$ as the congruence subgroup
\begin{equation*}
\Gamma_0(N):=\left\{\begin{pmatrix}
a& b\\
c& d
\end{pmatrix}\in SL_2(\Z):c\equiv 0\pmod N\right\}.
\end{equation*}
We recall Theorem 1.64 from~\cite[p. 18]{ono} regarding the modularity of eta-quotients:
\begin{prop}\label{mod}
If $f(z)=\prod_{\delta\mid N}\eta(\delta z)^{r_{\delta}}$ has integer weight $k=\frac{1}{2}\sum_{\delta\mid N}r_{\delta}$, with the additional properties that
\begin{equation*}
\sum_{\delta\mid N}\delta r_{\delta}\equiv 0\pmod{24}
\end{equation*}
and
\begin{equation*}
\sum_{\delta\mid N}\frac{N}{\delta}r_{\delta}\equiv 0\pmod{24},
\end{equation*}
then $f(z)$ satisfies
\begin{equation}\label{func}
f\left(\frac{az+b}{cz+d}\right)=\chi(d)(cz+d)^kf(z)
\end{equation}
for every $\begin{pmatrix}
a& b\\
c& d
\end{pmatrix}\in\Gamma_0(N)$ where the character $\chi$ is defined by $\chi(d):=\left(\frac{(-1)^ks}{d}\right)$, where $s:=\prod_{\delta\mid N}\delta^{r_\delta}$.
\end{prop}

Any modular form that is holomorphic (resp. vanishes) at all cusps of $\Gamma_0(N)$ and satisfies~(\ref{func}) is said to have Nebentypus character $\chi$, and the space of such forms is denoted $M_k(\Gamma_0(N),\chi)$ (resp. $S_k(\Gamma_0(N),\chi)$). In particular, if $k$ is a positive integer and $f(z)$ satisfies the conditions of Proposition~\ref{mod} and is holomorphic (resp. vanishes) at all of the cusps of $\Gamma_0(N)$, then $f(z)\in M_k(\Gamma_0(N),\chi)$ (resp. $S_k(\Gamma_0(N),\chi)$).  If $f(z)$ satisfies the conditions of Proposition~\ref{mod} but has poles at the cusps of $\Gamma_0(N)$, then we say $f(z)$ is a \textit{weakly holomorphic modular form}, and the space of such forms is denoted $M_k^{!}(\Gamma_0(N),\chi)$.

If $f(z)$ is a modular form, then we can act on it with various types of operators.  If $\gamma=\begin{pmatrix}
a & b\\
c & d
\end{pmatrix}
\in\SL_2(\Z)$, then the action of the \textit{slash operator} $\mid_{k}$ on $f(z)$ is defined by
\begin{equation*}
(f\mid_{k}\gamma)(z):=(cz+d)^{-k}f(\gamma z),
\end{equation*}
where
\begin{equation*}
\gamma z:=\frac{az+b}{cz+d}.
\end{equation*}

Next, we define
\begin{equation}\label{sigma}
\sigma_{v,t}:=\begin{pmatrix}
1 & v\\
0 & t
\end{pmatrix}.
\end{equation}
If $f(z)=\sum_{n=n_0}^{\infty}a(n)q^n$ is a weight $k$ modular form, then the action of the \textit{U-operator} $U(d)$ on $f(z)$ is defined by
\begin{equation*}\label{U}
f(z)\mid U(d):=t^{\frac{k}{2}-1}\sum_{v=0}^{t-1}f(z)\mid_{k}\sigma_{v,t}=\sum_{n=n_0}^{\infty}a(dn)q^n.
\end{equation*}
Likewise, the action of the \textit{V-operator} $V(d)$ is defined by
\begin{equation*}
f(z)\mid V(d):=t^{-\frac{k}{2}}f(z)\mid_{k}\begin{pmatrix}
t & 0\\
0 & 1
\end{pmatrix}
=
\sum_{n=n_0}^{\infty}a(n)q^{dn}.
\end{equation*}
Now, if $f(z)=\sum_{n=0}^{\infty}a(n)q^n\in M_k(\Gamma_0(N),\chi)$, then the action of the \textit{Hecke operator} $T_{p,k,\chi}$ on $f(z)$ is defined by
\begin{equation*}
f(z)\mid T_{p,k,\chi}:=\ds\sum_{n=0}^{\infty}(a(pn)+\chi(p)p^{k-1}a(n/p))q^n,
\end{equation*}
where $a(n/p)=0$ if $p\nmid n$.  We recall the following result from~\cite[p. 21,28]{ono}:
\begin{prop}\label{hecke}
Suppose that $f(z)\in M_k(\Gamma_0(N),\chi)$.
\begin{enumerate}
\item If $d\mid N$, then
\begin{equation*}
f(z)\mid U(d)\in M_k(\Gamma_0(N),\chi).
\end{equation*}
\item If $d$ is a positive integer, then
\begin{equation*}
f(z)\mid V(d)\in M_k(\Gamma_0(Nd),\chi).
\end{equation*}
\item If $p\geq 2$, then
\begin{equation*}
f(z)\mid T_{p,k,\chi}\in M_k(\Gamma_0(N),\chi).
\end{equation*}
\end{enumerate}
\end{prop}

\subsection{Ramanujan's Congruences and Atkin's Generalizations}\label{atkinsthm}

Ramanujan conjectured the following congruences for the partition function modulo powers of the primes 5 and 7, which Watson proved in~\cite{watson}.

\begin{thm}[Ramanujan]\label{ram}
Let $\ell=5$, $7$, or $11$ and let $j\geq 1$.  Then if $24n\equiv 1\pmod{\ell^j}$, we have that
\begin{equation*}
\begin{cases}
p(n)\equiv 0\pmod{\ell^j}& \ell=5,11\\
p(n)\equiv 0\pmod{\ell^{\lfloor{j/2}\rfloor+1}}& \ell=7.
\end{cases}
\end{equation*}
\end{thm}

In~\cite{atkin}, Atkin generalized the Ramanujan congruences modulo powers of 5, 7, and 11 to the function $p_k(n)$, which counts the number of $k$-colored partitions of $n$.

\begin{thm}[Atkin]\label{atkin}
Let $k>0$, $\ell=2,3,5,7$ or $13$, and $j\geq 1$.   Then if $24n\equiv k\pmod{\ell^j}$, we have that
\begin{equation*}
p_k(n)\equiv 0\pmod{\ell^{\lfloor{\alpha j/2+\epsilon}\rfloor}},
\end{equation*}
where $\epsilon:=\epsilon(k)=O(\log k)$ and $\alpha=\alpha(k,\ell)$ depending on $\ell$ and the residue of $k$ modulo 24.
\end{thm}
Atkin computes the value of $\alpha(k,\ell)$ in a table in~\cite{atkin}.  We note the following values of $\alpha$:
\begin{align}
\alpha(2,5)&=\alpha(2,7)=1.
\end{align}
In addition, following Atkin's method to calculate $\epsilon$ exactly, we observe
\begin{equation}
\epsilon=\begin{cases}
1-\lfloor{\log(48)/\log(\ell)}\rfloor=-1 & \ell=5\\
-\lfloor{\log(48)/\log(\ell)}\rfloor=-1 & \ell=7.
\end{cases}
\end{equation}
Therefore, for the case where $k=2$ and $\ell=5$ or 7, we have that for all $24n\equiv k\pmod{\ell^j}$,
\begin{equation}\label{57}
p_2(n)\equiv 0\pmod{\ell^{\lfloor{j/2-1}\rfloor}}.
\end{equation}

\subsection{Serre's Theorem}\label{serresthm}
We make use of the following theorem of Serre's regarding congruences for certain types of modular forms, stated in~\cite{treneer}:

\begin{thm}[Serre]\label{serre}
Suppose that $f(z)=\sum_{n=1}^{\infty}a(n)q^n\in S_k(\Gamma_0(N),\chi)$ has coefficients in the ring of integers $\mathcal{O}_K$ of a number field $K$ and $M$ is a positive integer.  Furthermore, suppose that $k\geq1$.  Then a positive proportion of the primes $p\equiv -1\pmod{MN}$ have the property that
\begin{equation*}
f(z)\mid T_{p,k,\chi}\equiv 0\pmod{M}.
\end{equation*}
\end{thm}
Serre's theorem guarantees the existence of congruences for cusp forms with coefficients in the ring of integers of a number field, which we will use to prove properties of the coefficients of the conjugacy growth series for $(\Alt(\N),S')$ and $(\Sym(\N),S)$.

\section{Proof of Theorem~\ref{cong1}}\label{pf}

We first prove congruences for arbitrary powers of $\ell=5$ or 7.  Let $j\geq 1$ and suppose that $24n\equiv 1\pmod{\ell^j}$.  Then by Theorem~\ref{ram}, we have that
\begin{equation}\label{P1}
\gamma_{\Sym(\N),S}(n)=p(n)\equiv0\pmod{\ell^{\lfloor{j/2}\rfloor+1}}.
\end{equation}
Additionally, we have that $24(2n)\equiv 2\pmod{\ell^j}$.  Using the case of Theorem~\ref{atkin} where $k=2$ and $\ell=5$ or 7, as in~(\ref{57}), we have that
\begin{equation}\label{P2}
p_2(2n)\equiv0\pmod{\ell^{\lfloor{j/2-1}\rfloor}}.
\end{equation}  
Therefore, for all $24n\equiv 1\pmod{\ell^j}$, we obtain from~(\ref{calt2}),~(\ref{P1}), and~(\ref{P2}) that
\begin{equation}
\gamma_{\Alt(\N),S'}(2n)\equiv\gamma_{\Sym(\N),S}(n)\equiv0\pmod{\ell^{\lfloor{j/2-1}\rfloor}},
\end{equation}
as desired.
\qed

\section{Congruences for $p_{2}(n)$}\label{proof of thm}

Because of the relationship between the conjugacy growth series for $(\Alt(\N),S')$ and the function $p_{2}(n)$, here, we focus on congruences for $p_{2}(n)$.  By the definition of the 2-colored partition function $p_{2}(n)$, we have that
\begin{align}\label{p2}
\sum_{n=0}^{\infty}p_{2}(n)q^n&=\prod_{n=1}^{\infty}\frac{1}{(1-q^n)^2}=\frac{q^{\frac{1}{12}}}{\eta^2(z)}.
\end{align}
Throughout, we let
\begin{equation}\label{f}
f(z):=\frac{1}{\eta(12z)^2}=\sum_{n=-1}^{\infty} a(n)q^n.
\end{equation}
Then we have that $p_{2}\left(\frac{n+1}{12}\right)=a(n)$.  In order to prove congruences between the coefficients of the conjugacy growth series for $(\Alt(\N),S')$ and $(\Sym(\N),S)$, we first prove a theorem concerning the coefficients of $f(z)$.  This makes effective the following result of Treneer~\cite{treneer} by determining the exact value of $m$ that is sufficiently large.

\begin{prop}[Treneer]
Suppose that $\ell$ is an odd prime and that $k$ and $m$ are integers.  Let $N$ be a positive integer with $(N,p)=1$, and let $\chi$ be a Dirichlet character modulo $N$.  Let $K$ be an algebraic number field with ring of integers $\mathcal{O}_K$, and suppose $f(z)=a(n)q^n\in M^{!}_{k}(\Gamma_0(N),\chi)\cap\mathcal{O}_K((q))$.  If $m$ is sufficiently large, then for each positive integer $j$, a positive proportion of the primes $Q\equiv -1\pmod{N\ell^j}$ have the property that
\begin{equation*}
a(Q\ell^m n)\equiv 0\pmod{\ell^j}
\end{equation*}
for all $n$ coprime to $Q\ell$.
\end{prop}

This section closely follows Section 3 in~\cite{treneer}.  Throughout this section, let $f(z)$ be defined by~(\ref{f}).  

\begin{thm}\label{cong}
Let $\ell\geq 5$ be prime and $j\in\N$.  Then for a positive proportion of primes $Q\equiv -1\pmod{144\ell^j}$, we have that
\begin{equation*}
a(Q\ell^{m_{\ell}} n)\equiv 0\pmod{\ell^j}
\end{equation*}
for all $n$ coprime to $Q\ell$.
\end{thm}

The proof of Theorem~\ref{cong} requires the construction of a cusp form that preserves congruence properties of the function $f(z)$.

\begin{prop}\label{cusp}
For every positive integer $j$, there exists an integer $\beta\geq j-1$ and a cusp form
\begin{equation*}
g_{\ell,j}(z)\in S_{\kappa}(\Gamma_0(144\ell^2),\chi)\cap\Z((q)),
\end{equation*}
where $\kappa:=-1+\frac{\ell^{\beta}(\ell^2-1)}{2}$, with the property that
\begin{equation*}
g_{\ell,j}(z)\equiv\sum_{n=1}^{\infty}a(\ell^{m_\ell}n)q^n\pmod{\ell^j}.
\end{equation*}
\end{prop}

We first require the following proposition concerning the Fourier expansion of $f(z)$ at a given cusp after being acted on by the $U(\ell^m)$ operator for $m\geq 1$.

\begin{prop}\label{expans}
Let $\gamma:=\begin{pmatrix}
a & b\\
c\ell^2 & d
\end{pmatrix}
\in\SL_2(\Z)$ where $c\in\Z$ and $ac>0$.  Then there exists an integer $n_0\geq -24$ and a sequence $\{a_0(n)\}_{n\geq n_0}$ such that for each $m\geq 1$, we have that
\begin{equation*}
(f(z)\mid U_{\ell^m})\mid_{-1}\gamma=\sum_{\substack{n=n_0\\n\equiv 0\pmod{\ell^m}}}^{\infty}a_0(n)q_{24\ell^m}^{n},
\end{equation*}
where $q_{24\ell^m}:=e^{\frac{2\pi iz}{24\ell^m}}$.
\end{prop}

The proof of this proposition makes use of the following lemma, which relies on the proof of Theorem 1 in~\cite{honda}.

\begin{lemma}\label{honda}
Given any matrix $A\in\SL_2(\Z)$, we have that
\begin{equation*}
f(z)\mid_{-1}A=\sum_{n=n_0}^{\infty}a_0(n)q_{24}^{n}
\end{equation*}
where $a_0(n)\in\Z$ and $n_0\geq -24$.
\end{lemma}
\begin{proof}
Let $A:=
\begin{pmatrix}a & b\\ c& d\end{pmatrix}\in\SL_2(\Z)$.  Then, as in the proof of Theorem 1 in ~\cite{honda}, we can write
\begin{equation}
\begin{pmatrix}
12 & 0\\
0 & 1
\end{pmatrix}
\begin{pmatrix}
a & b\\
c& d
\end{pmatrix}
=
\begin{pmatrix}
a' & b'\\
c' & d'
\end{pmatrix}
\begin{pmatrix}
\alpha & \beta\\
0 & \delta
\end{pmatrix}
\end{equation}
where $\begin{pmatrix}
a' & b'\\
c' & d'
\end{pmatrix}\in\SL_2(\Z)$, $\alpha,\beta,\delta\in\Z$, and $\alpha,\delta>0$.  Then we have that $12a=a'\alpha$ and $c=c'\alpha$, so $\alpha=(a'\alpha, c'\alpha)=(12a,c)=(12,c)\leq 12$.  Again, by Theorem 1 in~\cite{honda}, we obtain
\begin{equation}
f(z)\mid_{-1}A=\sum_{n=n_0}^{\infty}a_0(n)q_{24}^{n}
\end{equation}
where $n_0:=\frac{-2\alpha}{\delta}>-2\alpha\geq -24$.
\end{proof}

\begin{proof}[Proof of Proposition~\ref{expans}]
As in~\cite{treneer}, for each $0\leq v\leq \ell^m-1$, choose an integer $s_v$ such that
\begin{equation}
s_vN\equiv (a+vc\ell^2)^{-1}(b+vd)\pmod{\ell^m}
\end{equation}
and define $w_v:=s_vN$.  We let 
\begin{equation}
\alpha_0:=\begin{pmatrix}
a & 0\\
c\ell^{m+2} & d-w_0c\ell^2
\end{pmatrix}.
\end{equation}
By~(\ref{U}), we have that
\begin{equation}\label{eq4}
(f(z)\mid U_{\ell^m})\mid_{-1}\gamma=(\ell^m)^{-\frac{3}{2}}\sum_{v=0}^{\ell^m-1}f(z)\mid_{-1}\sigma_{v,\ell^m}\gamma.
\end{equation}
We observe that $\sigma_{v,\ell^m}\gamma=\alpha_0\sigma_{w_v,\ell^m}$, so we have that
\begin{equation}
(f(z)\mid U_{\ell^m})\mid_{-1}\gamma=(\ell^m)^{-\frac{3}{2}}\sum_{v=0}^{\ell^m-1}f(z)\mid_{-1}\alpha_0\sigma_{w_v,\ell^m}.
\end{equation}
By Lemma~\ref{honda}, we have that 
\begin{equation*}
f(z)\mid_{-1}\alpha_0=\sum_{n=n_0}^{\infty}a_0(n)q_{24}^{n},
\end{equation*} 
so we obtain 
\begin{align}\label{eq3}
\sum_{v=0}^{\ell^m-1}f(z)\mid_{-1}\alpha_0\sigma_{w_v,\ell^m}&=\sum_{v=0}^{\ell^m-1}\ell^{\frac{m}{2}}\sum_{n=n_0}^{\infty}a_0(n)e^{\frac{2\pi in(z+w_v)}{24\ell^m}}\\
\nonumber&=\ell^{\frac{m}{2}}\sum_{n=n_0}^{\infty}a_0(n)q_{24\ell^m}^n\sum_{v=0}^{\ell^m-1}e^{\frac{2\pi inw_v}{24\ell^m}}.
\end{align}
By Lemma 3.3 in~\cite{treneer}, the numbers $\frac{w_v}{24}$ run through the residue classes modulo $\ell^m$ as $v$ does.  Therefore, we have that
\begin{equation}\label{eq2}
\sum_{v=0}^{\ell^m-1}e^{\frac{2\pi inw_v}{24\ell^m}}=\sum_{v=0}^{\ell^m-1}e^{\frac{2\pi inv}{\ell^m}}=\begin{cases}
\ell^m & n\equiv 0\pmod{\ell^m}\\
0 &\text{else.}
\end{cases}
\end{equation}
Combining~(\ref{eq3}) and~(\ref{eq2}), we have that
\begin{equation}\label{eq1}
\sum_{v=0}^{\ell^m-1}f(z)\mid_{-1}\alpha_0\sigma_{w_v,\ell^m}=\ell^{\frac{3}{2}}\sum_{\substack{n=n_0\\n\equiv 0\pmod{\ell^m}}}^{\infty}a_0(n)q_{24\ell^m}^{n}.
\end{equation}
Using~(\ref{eq4}) and~(\ref{eq1}), we obtain
\begin{equation}
(f(z)\mid U_{\ell^m})\mid_{-1}\gamma=\sum_{\substack{n=n_0\\n\equiv 0\pmod{\ell^m}}}^{\infty}a_0(n)q_{24\ell^m}^{n},
\end{equation}
the Fourier expansion of $f(z)\mid U_{\ell^m}$ at the cusp $\frac{a}{c\ell^2}$.
\end{proof}

We now construct a weakly holomorphic modular form which vanishes at certain cusps of $\Gamma_0(144\ell^2)$.

\begin{prop}\label{vanish}
For each nonnegative integer $m$, define
\begin{equation*}
f_m(z):=f(z)\mid U_{\ell^m}-f(z)\mid U_{\ell^m+1}\mid V_{\ell}\in M_{-1}^{!}(\Gamma_0(144\ell^2),\chi).
\end{equation*}
Then $f_{m_\ell}$ vanishes at each cusp $\frac{a}{c\ell^2}$ of $\Gamma_0(144\ell^2)$ with $ac>0$.
\end{prop}

\begin{proof}
By Proposition~\ref{expans}, we have that
\begin{equation}
(f(z)\mid U_{\ell^{m_\ell}})\mid\gamma=\sum_{\substack{n= n_0\\n\equiv 0\pmod{\ell^{m_\ell}}}}^{\infty}a_0(n)q_{24\ell^{m_\ell}}^{n}
\end{equation}
where $n_0\geq-24$.  We now consider two cases.  If $5\leq\ell\leq 23$, we have that
\begin{equation*}
-\ell^{m_\ell}\leq -25<-24\leq n_0,
\end{equation*}
and if $\ell\geq 29$, we have that
\begin{equation*}
-\ell^{m_\ell}\leq -29<-24\leq n_0.
\end{equation*}
Suppose $a_0(n)\neq 0$.  Then $n\geq n_0>-\ell^{m_{\ell}}$, but $n\equiv 0\pmod{\ell^{m_\ell}}$, so $n\geq 0$.  Therefore, we obtain
\begin{equation}
(f(z)\mid U_{\ell^{m_\ell}})\mid\gamma=\sum_{\substack{n=0\\n\equiv 0\pmod{\ell^{m_{\ell}}}}}^{\infty}a_0(n)q_{24\ell^{m_\ell}}^{n}
\end{equation}
so $f(z)\mid U_{\ell^{m_{\ell}}}$ is holomorphic at the cusp $\frac{a}{c\ell^2}$.

Now, by Proposition 3.5 in~\cite{treneer}, we have that
\begin{equation}
f_m(z)\mid_{-1}\gamma=\sum_{\substack{n=0\\n\equiv 0\pmod{\ell^m}}}^{\infty}a_0(n)q_{24\ell^m}^n-\sum_{\substack{n=0\\n\equiv 0\pmod{\ell^{m+1}}}}^{\infty}a_0(n)q_{24\ell^m}^n,
\end{equation}
so the constant term in each expansion is $a_0(0)$, and they cancel.  Therefore, $f_{m_{\ell}}$ vanishes at the cusp $\frac{a}{c\ell^2}$.
\end{proof}

We are now ready to prove Proposition~\ref{cusp}.

\begin{proof}[Proof of Proposition~\ref{cusp}]
As in~\cite{treneer}, we define the eta-quotient
\begin{equation}
F_{\ell}(z):=\frac{\eta^{\ell^2}(z)}{\eta(\ell^2 z)}\in M_{\frac{\ell^2-1}{2}}(\Gamma_0(\ell^2)).
\end{equation}
By Theorem 1.65 in~\cite{ono}, we see that $F_{\ell}$ vanishes at every cusp $\frac{a}{c}$ of $\Gamma_0(144\ell^2)$ with $\ell^2\nmid c$.  We also have that $F_{\ell}(z)^{\ell^{s-1}}\equiv 1\pmod{\ell^s}$ for any integer $s\geq 1$.

Now, define
\begin{equation}
g_{\ell,j}(z):=f_{m_\ell}(z)\cdot F_{\ell}(z)^{\ell^\beta}
\end{equation}
where $\beta\geq j-1$ is sufficiently large such that $g_{\ell,j}(z)$ vanishes at all cusps $\frac{a}{c}$ of $\Gamma_0(144\ell^2)$ where $\ell^2\nmid c$.  By Theorem 1.65 in~\cite{ono}, it is possible to choose such a $\beta$ such that the order of vanishing of $g_{\ell,j}(z)$ is at least one at all such cusps.  Then $g_{\ell,j}\in\Z((q))$ and
\begin{equation}
g_{\ell,j}(z)\equiv f_{m_\ell}(z)\pmod{\ell^j}.
\end{equation}
By our choice of $\beta$, $g_{\ell,j}(z)$ vanishes at all cusps $\frac{a}{c}$ of $\Gamma_0(144\ell^2)$ where $\ell^2\nmid c$.
Furthermore, by Proposition~\ref{vanish}, $g_{\ell,j}(z)$ vanishes at all cusps $\frac{a}{c}$ where $\ell^2\mid c$.  
Define $\kappa:=-1+\frac{\ell^{\beta}(\ell^2-1)}{2}$.  Then we have that
\begin{equation}
g_{\ell,j}(z)\in S_{\kappa}(\Gamma_0(144\ell^2),\chi).
\end{equation}
By definition of $f_{m_{\ell}}$, we obtain
\begin{equation}
g_{\ell,j}(z)\equiv \sum_{n=1}^{\infty}a(\ell^{m_\ell}n)q^n-\sum_{n=1}^{\infty}a(\ell^{m_\ell+1}n)q^n\equiv\sum_{\substack{n=1\\\ell\nmid n}}^{\infty}a(\ell^{m_\ell}n)q^n\pmod{\ell^j}.
\end{equation}
Thus $g_{\ell,j}$ satisfies the conditions of Proposition~\ref{cusp}.
\end{proof}

Now that we have constructed the necessary cusp form, we arrive at the proof of Theorem~\ref{cong}.

\begin{proof}[Proof of Theorem~\ref{cong}]
By Proposition~\ref{cusp}, we can construct a cusp form $g_{\ell,j}\in S_{\kappa}(\Gamma_0(144\ell^2),\chi)\in\Z((q))$ such that
\begin{equation}
g_{\ell,j}(z)\equiv\sum_{\substack{n=1\\{\ell\nmid n}}}^{\infty}a(\ell^{m_\ell}n)q^n\pmod{\ell^j}.
\end{equation}
By Theorem~\ref{serre}, for a positive proportion of the primes $Q\equiv -1\pmod{144\ell^{j+2}}$, we have that
\begin{equation}
g_{\ell,j}(z)\mid T_{Q,\kappa,\chi}(Q)\equiv 0\pmod{\ell^{j}}.
\end{equation}
We can then write $g_{\ell,j}(z)=\sum_{n=1}^{\infty}b(n)q^n$ to obtain
\begin{equation}\label{gheck}
g_{\ell,j}(z)\mid T_{Q,\kappa,\chi}=\sum_{n=1}^{\infty}\left(b(Qn)+\chi(Q)Q^{\kappa-1}b(n/Q)\right)q^n\equiv0\pmod{\ell^j}.
\end{equation}
If $(Q,n)=1$, then the coefficient of $q^n$ in~(\ref{gheck}) is $b(Qn)$, so
\begin{equation}
a(Q\ell^{m_{\ell}}n)\equiv b(Qn)\equiv 0\pmod{\ell^j}
\end{equation}
for all $n$ coprime to $Q\ell$.
\end{proof}

\section{Proof of Theorem~\ref{sym}}\label{section4}
We now make use of Theorem~\ref{cong} to prove congruences between the coefficients of the conjugacy growth series for $(\Alt(\N),S')$ and $(\Sym(\N),S)$.
\begin{proof}[Proof of Theorem~\ref{sym}]
By~(\ref{calt2}), it is enough to show that $p_{2}(2Q\ell^{m_{\ell}}n+2\delta_\ell)\equiv 0\pmod{\ell^j}$.  By~(\ref{p2}) and~(\ref{f}), we observe $p_{2}\left(\frac{n+1}{12}\right)=a(n)$, so it suffices to prove the existence of congruences for $a(n)$.

By Theorem~\ref{cong}, for a positive proportion of primes $Q\equiv -1\pmod{144\ell^j}$, we have that
\begin{equation}\label{pa}
p_{2}\left(\frac{Q\ell^{m_{\ell}}n+1}{12}\right)=a(Q\ell^{m_{\ell}}n)\equiv 0\pmod{\ell^j}
\end{equation}
for all $n$ coprime to $Q\ell$.  Defining $\delta_{\ell}$ and $\beta_{\ell}$ by~(\ref{delta_l}) and~\ref{beta_l}), respectively, we can rewrite the left-hand side of equation~(\ref{pa}) as
\begin{equation}
p_{2}\left(2Q\ell^{m_{\ell}}n+2\delta_\ell\right)
\end{equation}
for all $24n+\beta_{\ell}$ coprime to $Q\ell$.  Therefore, for a positive proportion of primes $Q\equiv -1\pmod{144\ell^j}$, we have that
\begin{equation}
p_{2}\left(2Q\ell^{m_{\ell}}n+2\delta_\ell\right)\equiv 0\pmod{\ell^j},
\end{equation}
so we obtain
\begin{equation*}
2\gamma_{\Alt(\N),S'}(2Q\ell^{m_{\ell}}n+2\delta_\ell)\equiv\gamma_{\Sym(\N),S}(Q\ell^{m_{\ell}}n+\delta_\ell)\pmod{\ell^j},
\end{equation*}
as desired.
\end{proof}

\bibliography{references}
\end{document}